\documentclass[preprint]{elsarticle}

\usepackage[english]{babel}
\usepackage[utf8]{inputenc}

\usepackage{amsmath}
\usepackage{amssymb}
\usepackage{amsthm}
\usepackage{natbib}
\usepackage{stmaryrd}

\newdefinition{definition}{Definition}

\newtheorem{theorem}{Theorem}
\newtheorem{corollary}[theorem]{Corollary}
\newtheorem{example}[theorem]{Example}
\newtheorem{lemma}[theorem]{Lemma}
\newtheorem{proposition}[theorem]{Proposition}
\newtheorem{remark}[theorem]{Remark}

\newenvironment{myproof}[1]{%
    \paragraph{%
        \textbf{Proof of Theorem \ref{#1}}
    }
}{\hfill$\square$}

\newcommand{\trace}[1]{\textup{tr}\,#1}
\newcommand{\divergence}[1]{\textup{div}\,#1}
\newcommand{\Rm}{\textup{Rm}}
\newcommand{\Ric}{\textup{Ric}}

\title{Rigidity results for quotient almost Yamabe solitons}

\author{Marcelo Bezerra Barboza}
\ead{marcelo.barboza@ifgoiano.edu.br}
\address{Rodovia Geraldo Silva Nascimento, Km 2.5, Uruta\'{i}, Goi\'{a}s,
Brasil, 75790-000.}

\author{Willian Isao Tokura}
\ead{willianisaotokura@hotmail.com}
\address{Universidade Federal de Goi\'{a}s, IME, 131, 74001-970, Goi\^{a}nia,
GO, Brazil.}

\author{Elismar Dias Batista}
\ead{elismardb@gmail.com}
\address{Universidade Federal de Goi\'{a}s, IME, 131, 74001-970, Goi\^{a}nia,
GO, Brazil.}

\author{Priscila Marques Kai}
\ead{priscila.kai@hotmail.com}
\address{Universidade Federal de Goi\'{a}s, INF, s/n, 74690-900, Goi\^{a}nia,
GO, Brazil.}

\begin{document}

\begin{abstract}
    In this paper we investigate the structure of certain solutions of the fully
    nonlinear Yamabe flow, which we call quotient almost Yamabe solitons because
    they extend quite naturally those called quotient Yamabe solitons. We then
    present sufficient conditions for a compact quotient almost Yamabe soliton
    to be either trivial or isometric with an Euclidean sphere. We also
    characterize noncompact quotient gradient almost Yamabe solitons satisfying
    certain conditions on both its Ricci tensor and potential function.
\end{abstract}


\maketitle

\section{Introduction and main results}\label{sec:intro}

The Yamabe flow
\begin{equation}\label{yflow}
    \dfrac{\partial g^t}{\partial t}=-(R_{g^t}-r_{g^t})g^t,\quad g^0=g_{0},
\end{equation}
where $ R_{g^t} $ is the scalar curvature of $ g^t $ and
\[
    r_{g^t}
    =
    \dfrac{
        {
            \displaystyle
            \int_MR_{g^t}dv_{g^t}
        }
    }{
        {
            \displaystyle
            \int_Mdv_{g^t}
        }
    },
\]
is the mean value of $ R_{g^t} $ along $ M^n $, was introduced by R. Hamilton
\cite{hamilton1988ricci} and has become one of the standard tools of recent
differential geometry. Yamabe solitons arise as self-similar solutions of
\eqref{yflow}.

\begin{definition}
    A solution $g^t$ of \eqref{yflow} is a Yamabe soliton if there exist a
    smooth function $\alpha:[0,\varepsilon)\rightarrow(0,\infty)$, $
    \varepsilon>0 $, and a $1$-parameter family $\{\psi_t\}$ of diffeomorfisms
    of $M^n$ such that
    \[
        g^t=\alpha(t)\psi_{t}^{\ast}(g_{0}),
        \quad\alpha(0)=1\quad\mbox{and}\quad\psi_0=id_M.
    \]
\end{definition}

One gets
\begin{equation}\label{eqeass}
    \frac{1}{2}\mathcal{L}_{X}g=\left(R_g-\lambda\right)g,
\end{equation}
by substituting $ g^t=\alpha(t)\psi_t^\ast(g_0) $ into \eqref{yflow} and
evaluating the resulting expression at $ t=0 $, where $ \mathcal{L}_Xg $ is the
Lie derivative of $ g $ with respect to the field $ X $ of directions associated
with the $ 1 $-parameter family $ \{ \psi_t \} $ and $ \lambda=\alpha'(0)+r_{g}
$. Equation \eqref{eqeass} is the fundamental equation of Yamabe solitons.
Since their beginning, a lot has been proved about the nature of Yamabe
solitons. For example, Chow \cite{chow1992yamabe} proved that
compact Yamabe solitons have constant scalar curvature (see also
\cite{di2008yamabe, hsu2012note}). Daskalopoulos and Sesum
\cite{daskalopoulos2013classification} proved that complete locally conformally
flat Yamabe solitons with positive sectional curvature are rotationally
symmetric and must belong to the conformal class of flat Euclidean space.

A new soliton is born if one replaces the scalar curvature in \eqref{yflow} by
functions of the higher order scalar curvatures. As is the case with any
generalization, it's hoped that one recovers the old objects as particular
instances of the new ones, while opening room for new and exciting fenomena to
happen. In what follows we give formal definitions and even before we state out
main results, we examine a few examples. We included a section containing the
lemmas that we have used in the text for the convenience of the reader and a
separate section with the proofs to our statements can be found right after it.

The Riemann curvature tensor $\Rm$ of $ (M^n,g) $ admits the following
decomposition
\[
    \Rm=W_g+A_{g}\owedge g,
\]
where $ W_g $ and $ A_g $ are the tensors of Weyl and Schouten, respectively,
and $\owedge$ is the Kulkarni-Nomizu product of $ (M^n,g) $. Recall that the
Schouten tensor is given by
\[
    A_g=\frac{1}{n-2}\left(\Ric_{g}-\frac{R_g}{2(n-1)}g\right).
\]
The $ \sigma_k $-curvature of $ g $ is defined as the $ k $-th elementary
symmetric function of the eigenvalues $\lambda_{1},\dots,\lambda_{n}$ of the
endomorphism $ g^{-1}A_g $, that is,
\begin{equation*}
    \sigma_{k}(g)=
    \sigma_{k}(g^{-1}A_{g})=
    \sum_{1\leqslant i_{1}<\dots<i_{k}\leqslant n}
    \lambda_{i_{1}}\cdots\lambda_{i_{k}},\quad1\leqslant k\leqslant n.
\end{equation*}
Here we set $\sigma_{0}(g)=1$ for convenience. A simple calculation shows that
$\sigma_{1}(g)=\frac{R_g}{2(n-1)}$ which indicates that the
$\sigma_{k}$-curvature is a reasonable substitute for the scalar curvature of $
(M^n,g) $ in \eqref{yflow}.

Guan and Guofang introduced \cite{guan2004geometric} the fully nonlinear flow
\begin{equation}\label{flow}
    \dfrac{\partial g^t}{\partial t}
    =
    -\left(
        \log\frac{\sigma_{k}(g^t)}{\sigma_{l}(g^t)}-\log r_{k,l}(g^t)
    \right)g^t,\quad
    g^0=g_{0},
\end{equation}
where
\[
    \log r_{k,l}(g^t)=
    \dfrac{
            {
                \displaystyle
                 \int_{M}
                     \sigma_{l}(g^t)\log\dfrac{\sigma_{k}(g^t)}{\sigma_{l}(g^t)}
                 dv_{g^t}
            }
    }{
        {\displaystyle\int_{M}\sigma_{l}(g^t)dv_{g^t}}
    },
\]
was defined as to make the flow preserve the quantities
\[
    \begin{array}{rcll}
        \mathcal{E}_l(g^t)
        &=&
        {\displaystyle\int_{M}\sigma_{l}(g^t)dv_{g^t}},
        &
        l\neq\frac{n}{2},
        \\
        &&&\\
        &=&
        -{
            \displaystyle
            \int_{0}^{1}dt\int_{M}u\sigma_{\frac{n}{2}}(g^t)dv_{g^t}
         },
        &
        l=\frac{n}{2},
    \end{array}
\]
where $u\in C^\infty(M)$, $g=e^{-2u}g_0$ and $g^t=e^{-2tu}g_0$. The convergence
of the fully nonlinear flow was then proved under certain conditions to be
satisfied by the eigenvalues of the Schouten tensor. The authors also provided
geometric inequalities such as the Sobolev-type inequality in case $0\leqslant
l<k<\frac{n}{2}$, the conformal quasimass-integral-type inequality for
$\frac{n}{2}\leqslant k\leqslant n$, $1\leqslant l < k$ and the
Moser–Trudinger-type inequality for $k=\frac{n}{2}$.

Bo \textit{et al.} \cite{bo2018k} presented quotient Yamabe solitons as
self-similar solutions of the flow \eqref{flow} and stated rigidity results for
the existance of such objects on top of locally conformally flat manifolds. For
example, it was shown that any compact and locally conformally flat manifold
with the structure of a quotient Yamabe soliton, where both $\sigma_k>0$ and
$\sigma_l>0$, must have constant quotient curvature
$\frac{\sigma_{k}}{\sigma_{l}}$. Also, for the so called gradient $k$-Yamabe
soliton ($ l=0 $) they proved that, for $ k>1 $, any compact gradient $k$-Yamabe
soliton with negative constant scalar curvature necessarily has constant
$\sigma_{k}$-curvature. Almost Yamabe solitons were introduced by Barbosa and
Ribeiro \cite{barbosa2013conformal} as generalizations of self-similar solutions
of the Yamabe flow. Essentially, they allowed the parameter $\lambda$ in
\eqref{eqeass} to be a function on $M$. The authors then stated rigidity results
for almost Yamabe solitons on compact manifolds.  We refer the reader to
\cite{barbosa2013conformal, de2020almost, pirhadi2017almost,
seko2019classification} for further information.

In \cite{catino2012global} Catino \textit{et al.} proposed the study of
conformal solitons. A conformal soliton is a Riemannian manifold $(M^n,g)$
together with a nonconstant function $ f\in C^\infty(M) $ satisfying
$\nabla^2f=\lambda g$ for some $ \lambda\in\mathbb{R} $. They provided
classification results according to the number of critical points of $f$. It
should be noticed that solitons of Yamabe, $k$-Yamabe and quotient Yamabe types
are examples of conformal solitons.

We introduce quotient almost Yamabe solitons in extension to the quotient Yamabe
solitons.

\begin{definition}\label{def1}
    A solution $ g^t $ of \eqref{flow} is a quotient almost Yamabe soliton if
    there exist a function $ \alpha:M\times[0,\varepsilon)\rightarrow(0,\infty)
    $, $ \varepsilon>0 $, and a $ 1 $-parameter family $ \{\psi_t\} $ of
    diffeomorfisms of $ M^n $ such that
    \[
        g^t=\alpha(x,t)\psi_{t}^{\ast}(g_{0}),\quad
        \alpha(\,\cdot\,,0)\equiv1\mbox{ on }M^n\quad\mbox{and}\quad\psi_0=id_M.
    \]
    Equivalently, $(M^n,g)$ is a quotient almost Yamabe soliton if there exists
    a pair $X\in \mathfrak{X}(M)$, $\lambda\in C^\infty(M)$ satisfying
    \begin{equation}\label{eqeas}
        \frac{1}{2}\mathcal{L}_{X}g=
        \left(
        \log\frac{\sigma_{k}}{\sigma_{l}}-\lambda
        \right)g,\quad
        \sigma_{k}\cdot\sigma_{l}>0.
    \end{equation}
\end{definition}

We will write the soliton in \eqref{eqeas} as $(M^{n}, g, X, \lambda)$ for the
sake of simplicity. Following the terminology already in use with almost Yamabe
solitons, a soliton $ (M^n,g,X,\lambda) $ will be called:
\begin{enumerate}[a.]
    \item \textit{expanding} if $ \lambda<0 $,
    \item \textit{steady} if $ \lambda=0 $,
    \item \textit{shrinking} if $ \lambda>0 $ and, finally,
    \item \textit{indefinite} if $ \lambda $ change signs on $ M^n $.
\end{enumerate}

\begin{definition}
    A quotient gradient almost Yamabe soliton is a quotient almost Yamabe
    soliton $(M^n,g,X,\lambda)$ such that $ X=\nabla f $ is the gradient field
    of a function $ f\in C^\infty(M) $.
\end{definition}

Since
\[
    \dfrac{1}{2}\mathcal{L}_{\nabla f}g=\nabla^2 f,
\]
it follows from \eqref{eqeas} that a quotient gradient almost Yamabe soliton $
(M^n,g,\nabla f,\lambda) $ is characterized by the equation
\begin{equation}\label{eq fundamental}
    \nabla^2 f=
    \left(
    \log\frac{\sigma_{k}}{\sigma_{l}}-\lambda
    \right)g,
    \quad
    \sigma_{k}\cdot\sigma_{l}>0.
\end{equation}
Quotient almost Yamabe solitons, gradient or not, are regarded as
\textit{trivial} if their defining equation vanishes identically. Thus, $
(M^n,g,X,\lambda) $ is trivial if $ \mathcal{L}_Xg=0 $ and $ (M^n,g,\nabla
f,\lambda) $ if $ \nabla^2 f=0 $. In either case, $
\log{\frac{\sigma_k}{\sigma_l}}-\lambda=0 $.

Before we state our main results, let's take a look at some examples.

\begin{example}\label{exemplo22}
    The product manifold
    $(\mathbb{R}\times\mathbb{S}^{n},g=dt^2+g_{\mathbb{R}^n})$ alongside the
    function
    \[
        f:\mathbb{R}\times\mathbb{S}^n\to\mathbb{R},\quad
        (t,x)\mapsto f(t,x)=at+b\quad(a,b\in\mathbb{R}),
    \]
    is, for $ k=l=1 $, a trivial quotient gradient almost Yamabe soliton with $
    \lambda=0 $, since $\sigma_{1}(g^{-1}A_{g})=\frac{n}{2}$ and $ \nabla^2
    f=0 $.
\end{example}

\begin{example}
    Identities
    \[
        \Ric_{g_{\mathbb{S}^{n}}}=(n-1)g_{\mathbb{S}^{n}},\quad
        R_{g_{\mathbb{S}^{n}}}=n(n-1)\quad\mbox{and}\quad
        A_{g_{\mathbb{S}^n}}=\frac{1}{2}g_{\mathbb{S}^n},
    \]
    rule the Ricci tensor, scalar curvature and Schouten tensor, respectively,
    of the Euclidean sphere $ (\mathbb{S}^n,g_{\mathbb{S}^n}) $. Therefore, we
    have that
    \[
        \sigma_{k}(g_{\mathbb{S}^{n}}^{-1}A_{g_{\mathbb{S}^{n}}})=
        \frac{1}{2^k}\binom{n}{k},\quad1\leqslant k\leqslant n.
    \]
    Consider the height function
    \[
        h_v:\mathbb{S}^n\rightarrow\mathbb{R},\quad
        x\mapsto h_{v}(x)=\langle x,v\rangle,
    \]
    on $ \mathbb{S}^n $ with respect to a given $ v\in\mathbb{S}^n $. It then
    follows that
    \[
        \nabla^{2}h_{v}=-h_{v}g_{g_{\mathbb{S}^n}},
    \]
    showing that $(\mathbb{S}^{n},g_{\mathbb{S}^{n}}, \nabla h_{v},\lambda)$ is
    a compact quotient almost Yamabe soliton with
    \[
        \lambda:\mathbb{S}^n\to\mathbb{R},\quad
        x\mapsto h_v(x)+\log{\frac{\sigma_k}{\sigma_l}}.
    \]
\end{example}

\begin{example}
    On the hyperbolic space $(\mathbb{H}^{n},g_{\mathbb{H}^{n}})$ one has
    \[
        \Ric_{g_{\mathbb{H}^{n}}}=-(n-1)g_{\mathbb{H}^{n}},\quad
        R_{g_{\mathbb{H}^{n}}}=-n(n-1)\quad\mbox{and}\quad
        A_{g_{\mathbb{H}^{n}}}=-\frac{1}{2}g_{\mathbb{H}^{n}},
    \]
    for its Ricci tensor, scalar curvature and Schouten tensor, respectevily.
    Therefore, we have that
    \[
        \sigma_{k}(g_{\mathbb{H}^{n}}^{-1}A_{g_{\mathbb{H}^{n}}})=
        \frac{(-1)^{k}}{2^k}\binom{n}{k},\quad1\leqslant k\leqslant n.
    \]
    We consider the model $\mathbb{H}^{n}=\{x\in\mathbb{R}^{n,1}: \langle
    x,x\rangle_{0}=-1, x_{1}>0\}$ of the hyperbolic space, where
    $\mathbb{R}^{n,1}$ is nothing but the Euclidean space $\mathbb{R}^{n+1}$
    endowed with lorentzian inner product $\langle
    x,x\rangle_{0}=-x_{1}^2+x_{2}^{2}+\dots+x_{n+1}^{2}$. As in our previous
    example, we consider the height function
    \[
        h_v:\mathbb{H}^n\rightarrow\mathbb{R},\quad
        x\mapsto h_{v}(x)=\langle x,v\rangle_{0},
    \]
    on $ \mathbb{H}^n $ with respect to a given $
    v\in\mathbb{H}^n$. Because
    \[
        \nabla^{2}h_{v}=h_{v}g_{g_{\mathbb{H}^n}},
    \]
    we conclude that $(\mathbb{H}^{n},g_{\mathbb{H}^{n}}, \nabla h_{v},\lambda)$
    is a quotient almost Yamabe soliton with
    \[
        \lambda:\mathbb{H}^n\to\mathbb{R},\quad
        x\mapsto-h_v(x)+\log{\frac{\sigma_k}{\sigma_l}},
    \]
    as long as we have $ k\equiv l\pmod2 $.
\end{example}

\begin{example}
    Consider $ \mathbb{R}^n $ endowed with a metric tensor of the form
    \[
        g_{ij}=e^{2u_i}\delta_{ij},\quad1\leqslant i,j\leqslant n,
    \]
    in cartesian coordinates $ x=(x_1,\ldots,x_n) $ of $ \mathbb{R}^n $ where $
    u_1,\ldots,u_n\in C^\infty(\mathbb{R}^n) $. Then, the Ricci tensor of $
    (\mathbb{R}^n,g) $ is ruled \cite{landau_classical_1975} by the formulas
    \[
        \begin{array}{rcl}
            \Ric_g(\partial_j,\partial_k)
            &=&
            \sum_{l\neq k,j}\,U_{jk}^l+u_{j,k}u_{l,j},
            \quad j\neq k,
            \\
            &&\\
            \Ric_g(\partial_k,\partial_k)
            &=&
            \sum_{l\neq k}\,
            e^{2(u_k-u_l)}U_{ll}^k+
            U_{kk}^l-
            \sum_{m\neq k,l}\,e^{2(u_k-u_m)}u_{k,m}u_{l,m},
        \end{array}
    \]
    where
    \[
        u_{i,j}=\dfrac{\partial u_i}{\partial x_j},\quad\mbox{and}\quad
        u_{i,j,k}=\dfrac{\partial^2u_i}{\partial x_k\partial x_j},\quad
        1\leqslant i,j,k\leqslant n,
    \]
    and $ U_{jk}^l=u_{l,k}(u_k-u_l)_{,j}-u_{l,j,k} $. Assume that $ 4\leqslant
    n\in\mathbb{Z} $. Also, let $ \tau $ be the $ n $-cycle $ (1,2,3,\ldots,n) $
    in the symmetric group $ S_n $. It turns out that by choosing
    \[
        \begin{array}{rcll}
            u_i(x_1,\ldots,x_n)
            &=&
            \log{\cosh{\left(x_{\tau(i)}\right)}},
            &
            i\equiv0\pmod2,
            \\
            &=&
            0,
            &
            i\equiv1\pmod2,
        \end{array}
    \]
    we simplify the situation a little bit and the Ricci tensor of $
    (\mathbb{R}^n,g) $ ends up being a constant multiple of the metric, $
    \Ric_g=-g $. Therefore, $ (\mathbb{R}^n,g) $ is a complete Einstein manifold
    and, as such, $ A_g=\frac{-1}{2(n-1)}g $. Then, we have that
    \[
        \sigma_k(g^{-1}A_g)=\dfrac{(-1)^k}{2^k(n-1)^k}\binom{n}{k},\quad
        1\leqslant k\leqslant n.
    \]
    Because $ X=(0,1,\ldots,0,1) $ is a Killing field on $ (\mathbb{R}^n,g) $ we
    know that $ (\mathbb{R}^n,g,X,\lambda) $ is a trivial quotient almost Yamabe
    soliton with whenever $ k\equiv l\pmod2 $. It should be mentioned that $ X $
    is not a gradient field with respect to the metric $ g $.
\end{example}

Any smooth vector field $ X $ on a compact Riemannian manifold $ (M^n,g) $ can
be written in the form
\begin{equation}\label{Hodge}
    X = \nabla h + Y,
\end{equation}
where $Y\in \mathfrak{X}(M)$ is divergence free and $h\in C^\infty(M)$. In fact,
by the Hodge-de Rham Theorem \cite{warner2013foundations} we have that
\[
    X^{\flat}=d\alpha+\delta\beta+\gamma.
\]
Now, take $Y = (\delta\beta+\gamma)^{\sharp},\,\nabla h=(d\alpha)^{\sharp}$ and
we are done. The function $ h $ is called the Hodge-de Rham potential of $ X $.
Our first theorem states the triviality of a compact quotient almost Yamabe
soliton under certain integral assumptions.

\begin{theorem}\label{Te1}
    A compact quotient almost Yamabe soliton $(M^n,g, X,\lambda)$ is trivial if
    one of the following is true:
    \begin{enumerate}[a)]
        \item
            ${\displaystyle\int_{M}e^{\lambda}\sigma_{l}\langle\nabla
            \lambda,X\rangle dv_{g}= -\int_{M}e^{\lambda}\langle\nabla
            \sigma_{l}, X\rangle dv_{g}}$ plus any of these:
            \begin{enumerate}[i.]
                \item
                    $\nabla\Ric_g=0$;
                \item
                    $\divergence{C_g}=0$ where $ C_g $ is the Cotton tensor of $
                    (M^n,g) $;
                \item
                    $X=\nabla f$ is a gradient vector field;
            \end{enumerate}
        \item
            ${\displaystyle\int_{M}\langle\nabla h,X\rangle dv_{g}\leqslant0}$
            where $h$ is the Hodge-de Rham potential of $ X $.
    \end{enumerate}
\end{theorem}

The next two corollaries deal with quotient Yamabe solitons ($ \lambda $ is a
real constant) and constitute direct applications of Theorem \ref{Te1}. In
\cite{bo2018k} Bo \textit{et al.} proved that $\sigma_k/\sigma_l$ must be
constant on any compact and locally conformally flat quotient Yamabe soliton.
We extend Bo's result.

\begin{corollary}\label{T3}
    Let $(M^n,g, X,\lambda)$ be any compact quotient Yamabe soliton with a null
    cotton tensor. Then, $\sigma_{k}/\sigma_{l}$ is constant and, as such, the
    soliton is trivial.
\end{corollary}

In \cite{catino2012global} Catino \textit{et al.} proved that any compact
gradient $k$-Yamabe soliton with nonnegative Ricci tensor is trivial. Bo
\textit{et al.} \cite{bo2018k} also proved that any compact gradient $k$-Yamabe
soliton with constant negative scalar curvature is trivial. In
\cite{tokura2020triviality} it was shown that any compact gradient $k$-Yamabe
soliton must be trivial. We extend all these results at once.

\begin{corollary}\label{T1}
    Let $(M^n,g, \nabla f,\lambda)$ be any compact quotient gradient Yamabe
    soliton. Then, $\sigma_k/\sigma_l$ is constant and, as such, the soliton is
    trivial.
\end{corollary}

Yet another triviality result hold for quotient almost Yamabe solitons if one
drops compacity on $ M^n $ in favor of a decaiment condition on the norm of the
soliton field $ X $.

\begin{theorem}\label{T4}
    Let $(M^n,g,X,\lambda)$ be a complete and noncompact quotient almost Yamabe
    soliton satisfying
    \begin{equation*}
        \int_{M^n\setminus B_{r}(x_{0})}\dfrac{|X|}{d(x,x_{0})}dv_{g}<\infty
        \quad\mbox{and}\quad\mathcal{L}_Xg\geqslant0,
    \end{equation*}
    where $d$ is the distance function with respect to $g$ and $B_{r}(x_{0})$ is
    the ball of radius $r>0$ centered at $x_{0}$. Then, $(M^n,g,X,\lambda)$ is
    trivial.
\end{theorem}

Next, we give a suficient condition for a compact quotient gradient almost
Yamabe soliton to be isometric with an Euclidean sphere.

\begin{theorem}\label{Teoo1}
    Let $(M^n,g, \nabla f,\lambda)$ be a nontrivial compact quotient gradient
    almost Yamabe soliton with constant scalar curvature $R_g=R>0$. Then
    $(M^n,g)$ is isometric to the Euclidean sphere $\mathbb{S}^{n}(\sqrt{r})$, $
    r=R/n(n-1) $. Moreover, up to a rescaling the potential $f$ is given by
    $f=h_{v}+c$ where $h_{v}$ is the height function on the sphere and $c$ is a
    real constant.
\end{theorem}

\begin{remark}
    A similar result concerning almost Ricci solitons is found in
    \cite{barros2016some}.
\end{remark}

Another situation in which a quotient gradient almost Yamabe soliton must be
isometric with an Euclidean sphere is described below.

\begin{theorem}\label{teosigmak}
    Let $(M^n,g, \nabla f,\lambda)$ be a nontrivial compact quotient gradient
    almost Yamabe soliton with constant $\sigma_{k}$-curvature, for some
    $k=2,\dots,n$, and $A_g>0$. Then, $(M^n,g)$ is isometric with an Euclidean
    sphere $\mathbb{S}^{n}$.
\end{theorem}

Finally we investigate the structure of noncompact quotient gradient almost
Yamabe solitons satisfying reasonable conditions on its potential function
and both Ricci and scalar curvatures.

\begin{theorem}\label{final}
    Let $(M^n,g, \nabla f,\lambda)$ be a nontrivial and noncompact quotient
    gradient  almost Yamabe soliton. Assume that
    \[
        \mathcal{L}_{\nabla f^2}R\geqslant0,\quad
        \overset{\circ}{\Ric}\geqslant0\quad\mbox{and}\quad
        |\overset{\circ}{\Ric}(\nabla f^2)|\in L^{1}(M).
    \]
    Then, $(M^n,g)$ has constant scalar curvature $R_g=R\leqslant0$ and $f$ has
    at most one critical point. Moreover, we have that:
    \begin{enumerate}[a)]
        \item
            If $R=0$, then $(M^n,g)$ is isometric with a Riemannian product
            manifold $(\mathbb{R}\times\mathbb{F}^{n-1},dt^2+g_{\mathbb{F}})$
            such that $ \Ric_{g_{\mathbb{F}}}\geqslant0 $; \vspace{0,1cm}
        \item
            If $R<0$ and $f$ has no critical points, then $(M^n,g)$ is isometric
            with a warped product manifold
            $(\mathbb{R}\times\mathbb{F}^{n-1},dt^2+\xi(t)^2g_{\mathbb{F}})$
            such that
            \[
                \xi''+\frac{R}{n(n-1)}\xi=0;
            \]
            \vspace{0,1cm}
        \item
            If $R<0$ and $f$ has only one critical point, then $(M^n,g)$ is
            isometric with a hyperbolic space.
    \end{enumerate}
\end{theorem}

\section{Key Lemmas}\label{sec:lemmas}

\begin{lemma}\label{KW1}
    \textup{(\cite{barbosa2020generalized, gover2013universal})}
    Let $ (M^n,g) $ be a compact Riemannian manifold with a possibly empty
    boundary $ \partial M $. Then,
    \[
        \int_{M}X(\trace{T})dv_{g}=
        n\int_{M}\divergence{T(X)}dv_{g}+
        \frac{n}{2}
        \int_{M}\langle \overset{\circ}{T},\mathcal{L}_{X}g\rangle dv_{g}-
        n\int_{\partial M}\overset{\circ}{T}(X,\nu)ds_{g},
    \]
    for every symmetric $ (0,2) $-tensor $ T $ and every vector field $ X $ on $
    M $, where
    \[
        \trace{T}=g^{ij}T_{ij}\quad\mbox{and}\quad
        \overset{\circ}{T}=T-\frac{\trace{T}}{n}g,
    \]
    and $\nu$ is the outward unit normal field on $ \partial M $.
\end{lemma}

\begin{proof}
    First notice that integration by parts yields
    \[
        \int_{\partial M}T(X,\nu)dA_{g}=\int_{M}\nabla^{i}(T_{ij}X^{j})dv_{g},
    \]
    and because
    \begin{align*}
        \nabla^{i}(T_{ij}X^{j})
        &=
        \nabla^{i}T_{ij}X^{j}+T_{ij}\nabla^{i}X^{j}
        \\ &=
        \nabla^{i}T_{ij}X^{j}+\frac{1}{2}T_{ij}(\nabla^{i}X^{j}+\nabla^{j}X^{i})
        \\ &=
        \divergence{T(X)}+\frac{1}{2}\langle T,\mathcal{L}_{X}g\rangle,
    \end{align*}
    we get that
    \begin{equation}\label{eq2222}
        \begin{split}
            \int_{\partial M}T(X,\nu)dA_{g}
            &=
            \int_{M}\nabla^{i}(T_{ij}X^{j})dv_{g}
            \\
            &=
            \int_{M} \divergence{T(X)} dv_{g}+
            \frac{1}{2}\int_{M}\langle T,\mathcal{L}_{X}g\rangle dv_{g}
            \\
            &=
            \int_{M} \divergence{T(X)} dv_{g}+
            \frac{1}{2}
            \int_{M}\langle \overset{\circ}{T},\mathcal{L}_{X}g\rangle dv_{g}+
            \frac{1}{2}
            \int_{M}\frac{\trace{T}}{n}\langle g,\mathcal{L}_{X}g\rangle dv_{g}
            \\
            &=
            \int_{M} \divergence{T(X)} dv_{g}+
            \frac{1}{2}
            \int_{M}\langle \overset{\circ}{T},\mathcal{L}_{X}g\rangle dv_{g}+
            \frac{1}{n}\int_{M}\trace{T}\cdot\divergence{X}dv_{g}.
        \end{split}
    \end{equation}
    On the other hand, we have that
    \begin{equation}\label{eq222}
        \int_{M}\trace{T}\cdot\divergence{X}dv_{g}=
        \int_{\partial M}
        \trace{T}\cdot\langle X,\nu\rangle dA_{g}-\int_{M}X(\trace{T})
        dv_{g}.
    \end{equation}
    The result now follows from \eqref{eq2222} and \eqref{eq222} above.
\end{proof}

\begin{lemma}\label{naza}
    \textup{(\cite{barros2013compact})} Let $ (M^n,g) $ be a Riemannian manifold
    and $ T $ be a symmetric $ (0,2) $-tensor field on $ M^n $. Then,
    \[
        \divergence(T(\varphi X))=
        \varphi(\divergence{T})(X)+
        \varphi \langle\nabla X, T\rangle+
        T(\nabla\varphi,X),
    \]
    for each $X\in\mathfrak{X}(M)$ and each $\varphi\in C^\infty(M)$ where $
    T(X) $ is the vector field $ g $-equivalent to $ T $.
\end{lemma}

For locally conformally flat manifolds, a proposition similiar to the next one
can be found in \cite{han2006kazdan}. Recall that a vector field $ X $ on a
Riemannian manifold $ (M^n,g) $ is a conformal field in case
\[
    \frac{1}{2}\mathcal{L}_Xg=\varphi g,
\]
for some $ \varphi\in C^\infty(M) $.
\begin{proposition}\label{lm}
    If $ X $ is a conformal vector field on a compact Riemannian manifold
    $(M^n,g)$ with null Cotton tensor, then
    \[
        \int_{M^n}\langle X,\nabla \sigma_{k}\rangle dv_{g}=0,
    \]
    for every $ k=1,2,\ldots,n $.
\end{proposition}

Recall that the $k$-Newton tensor field associated with $g^{-1}A_g$ is defined
by
\[
    T_{k}(g^{-1}A_g)=\sum_{j=0}^k\,(-1)^j\sigma_{k-j}(g)(g^{-1}A_g)^j,
    \quad1\leqslant k\leqslant n.
\]
Among the identities satisfied by $ T_k(g^{-1}A_g) $ one finds (see
\cite{barros2016some})
\[
    \trace{T_k}(g^{-1}A_g)=(n-k)\sigma_k(g)\quad\mbox{and}\quad
    \divergence{T_k}(g^{-1}A_g)=0,
\]
for every $\leqslant k\leqslant n$.

\begin{proof}
    Let $ \varphi\in C^\infty(M) $ be such that
    \[
        \frac{1}{2}\mathcal{L}_Xg=\varphi g,
    \]
    and take $ T_k=T_k(g^{-1}A_g) $ where $ k\in\{1,2,\ldots,n-1\} $. Now a
    direct application of Lemma \ref{KW1} yelds
    \begin{equation}\label{eqdd}
        \int_{M}X(\trace{T_k}) dv_{g}=
        n\int_{M} \divergence{T_k(X)}dv_{g}+
        n\int_{M}\varphi\langle \overset{\circ}{T_k},g\rangle dv_{g}.
    \end{equation}
    It follows from Corollary 1 of \cite{barros2016some} that $
    \divergence{T_k}=0 $ and because
    \[
        \overset{\circ}{T_k}
        =
        T_k-
        \frac{\trace{T_k}}{n}g
        =
        T_k-
        \frac{n-k}{n}\sigma_{k}g,
    \]
    equation \eqref{eqdd} can rewritten in the simpler form
    \[
        (n-k)\int_{M}\langle X,\nabla\sigma_{k}\rangle dv_{g}=0,
    \]
    which proves the proposition in case $k\neq n$. As for the remaining case,
    it follows from \cite{han2006kazdan} that
    \[
        n\langle X,\nabla\sigma_{n}\rangle=
        \nabla_{a}
        \left[
            T_{b}^{a}\nabla^{b}(\divergence{X})+2n\sigma_{n}X^a
        \right],
    \]
    where $T_{b}^{a}$ are the components of $T_{n-1}(g^{-1}A_g)$. Therefore,
    if we go there and write
    \[
        Y^{a}=T_{b}^{a}\nabla^{b}(\divergence{X})+2n\sigma_{n}X^a,
    \]
    we get that
    \[
        n\int_{M}\langle X,\nabla\sigma_{k}\rangle dv_{g}=
        \int_{M}\nabla_{a}Y^{a}dv_{g}=0,
    \]
    which proves the proposition also for $k=n$.
\end{proof}

Our next Lemma states structural equations for quotient gradient almost Yamabe
solitons.

\begin{lemma}\label{lema}
    Let $(M^n,g,\nabla f,\lambda)$ be a quotient gradient almost Yamabe soliton.
    Then, we have that:
    \begin{enumerate}[a)]
        \item
            $\Delta f=
            n\left(\log \dfrac{\sigma_{k}}{\sigma_{l}}-
            \lambda\right)$;
        \item
            $(n-1)\nabla\left(\log\dfrac{\sigma_{k}}{\sigma_{l}}-
            \lambda\right)+\Ric(\nabla f)=0$;
        \item
            $(n-1)\Delta\left(\log\dfrac{\sigma_{k}}{\sigma_{l}}-
            \lambda\right)+
            \dfrac{1}{2}\langle\nabla R,\nabla f\rangle+
            \left(\log\dfrac{\sigma_{k}}{\sigma_{l}}-
            \lambda\right)R=0$.
    \end{enumerate}
\end{lemma}

\begin{proof}
    \begin{enumerate}[a)]
        \item
            Simply take traces at equation \eqref{eq fundamental};
        \item
            Next, we differentiate \eqref{eq fundamental} to get
            \[
                \nabla_{j}\nabla_{r}\nabla_{i}f=
                \nabla_{j}\left(\log\frac{\sigma_{k}}{\sigma_{l}}-
                \lambda\right)g_{ri},
            \]
            from what we see that
            \[
                \nabla_{i}\nabla_{j}\nabla_{r}f+
                \sum_{s}R_{rijs}\nabla_{s}f=
                \nabla_{j}\left(\log\frac{\sigma_{k}}{\sigma_{l}}-
                \lambda\right)g_{ri},
            \]
            with the help of the Ricci identity that can be found in
            (\cite{aubin2013some}, pg. 4). Now we only need to contract this
            equation on the indices $j,r$ in order to get
            \[
                \nabla_{i}\Delta f+
                \sum_{s}Ric_{is}\nabla_{s}f=
                \nabla_{i}\left(\log\frac{\sigma_{k}}{\sigma_{l}}-
                \lambda\right),
            \]
            then yelding
            \begin{equation}\label{derivada}
                (n-1)\nabla_{i}\left(\log\frac{\sigma_{k}}{\sigma_{l}}-
                \lambda\right)+\sum_{s}Ric_{is}\nabla_{s}f=0,
            \end{equation}
            by $ \textup{a)} $, which proves $\textup{b)}$;
        \item
            As for the remaining identity, we apply the divergence operator on
            both sides of \eqref{derivada} and use the twice contracted second
            Bianchi identity to obtain
            \[
                (n-1)\Delta
                \left(
                \log\frac{\sigma_{k}}{\sigma_{l}}-\lambda
                \right)+
                \frac{1}{2}\langle\nabla R,\nabla f\rangle+
                \sum_{sl}Ric_{sl}\nabla_{s}\nabla_{l}f=0,
            \]
            which is equivalent to
            \[
                (n-1)\Delta
                \left(
                \log\frac{\sigma_{k}}{\sigma_{l}}-\lambda
                \right)+
                \frac{1}{2}\langle\nabla R,\nabla f\rangle+
                \left(
                \log\frac{\sigma_{k}}{\sigma_{l}}-\lambda
                \right)R
                =0,
            \]
            by the fundamental equation \eqref{eq fundamental}, which proves $
            \textup{c)} $.
    \end{enumerate}
    This concludes the proof.
\end{proof}

\section{Proofs of the main results}\label{sec:proofs}

\begin{myproof}{Te1}
    \begin{enumerate}[a)]
        \item
            We get
            \[
                \int_{M}\Ric_{jk}\nabla_{i}C_{ijk}dv_{g}=
                -\int_{M}\nabla_{i}\Ric_{jk}C_{ijk}dv_{g}=0,
            \]
            if either $\nabla\Ric_g=0$ or $\divergence{C_g}=0$ and because
            \begin{equation}
                \begin{split}
                    \int_{M}\nabla_{i}\Ric_{jk}C_{ijk}dv_{g}
                    &=
                    \int_{M}
                    \left[
                        C_{ijk}+
                        \frac{1}{2(n-1)}
                        \left(
                        g_{jk}\nabla_{i}R_g-g_{ij}\nabla_{j}R_g
                        \right)
                        \right]C_{ijk}dv_{g}
                    \\
                    &=
                    \int_{M}|C_g|^{2}dv_{g}+
                    \frac{1}{2(n-1)}
                    \int_{M}
                    \left(
                    C_{ijk}g_{jk}\nabla_{i}R_g-
                    C_{ijk}g_{ij}\nabla_{j}R_g
                    \right)dv_{g}
                    \\
                    &=\int_{M}|C_g|^{2}dv_{g},
                \end{split}
            \end{equation}
            we conclude that $C_g=0$. Equation \eqref{eq fundamental} implies
            that $ X $ is a conformal field and so we can apply Proposition
            \ref{lm} to conclude that
            \begin{equation*}
                \int_{M^n}
                \sigma_{k}\left(\log\frac{\sigma_{k}}{\sigma_{l}}-\lambda\right)
                dv_{g}
                =
                -\dfrac{1}{n}
                \int_{M^n}
                \langle\nabla\sigma_{k},X\rangle
                dv_{g}
                =0.
            \end{equation*}
            Therefore, we have that
            \begin{equation}\label{eqeq}
                \begin{split}
                    \int_{M^n}
                    \frac{\sigma_{l}}{n}
                    \left(
                    \frac{\sigma_{k}}{\sigma_{l}}-e^\lambda
                    \right)
                    \left(
                    \log\frac{\sigma_{k}}{\sigma_{l}}-\lambda
                    \right)
                    dv_{g}
                    &=
                    -\int_{M^n}
                    \frac{e^{\lambda}\sigma_{l}}{n}
                    \left(
                    \log\frac{\sigma_{k}}{\sigma_{l}}-\lambda
                    \right)
                    dv_{g}
                    \\
                    &=
                    \int_{M^n}
                    e^{\lambda}\sigma_{l}\langle \nabla \lambda, X\rangle
                    dv_{g}+
                    \int_{M^n}
                    e^{\lambda}\langle \nabla \sigma_{l}, X\rangle
                    dv_{g}=0,
                \end{split}
            \end{equation}
            by our hypothesis on the nullity of the integral at the right hand
            of \eqref{eqeq}. Because $\sigma_{l}\neq0$ does not change signs on
            $ M^n $ we must then admit that $\sigma_{k}/\sigma_{l}=e^\lambda$,
            which proves our assertion in case $ \Ric_g $ is parallel or $ C_g $
            is divergence free. Instead, if $ X=\nabla f $, we argue by
            contradiction to show that $ f $ is a constant function. Should $f$
            not be constant on $ M^n $, the manifold $(M^n,g)$ could not lie in
            any conformal class other than that of the Euclidean sphere
            $(\mathbb{S}^n,g_{\mathbb{S}^n})$, by Theorem 1.1 of
            \cite{catino2012global}. So, just as it happens with any locally
            conformally flat manifold, the Cotton tensor of $(M^n,g)$ would then
            vanish identically and by what has been said above $ (M^n,g,\nabla
            f,\lambda) $ ought to be trivial. This contradiction shows that $ f
            $ is indeed a constant function, now concluding \textup{a)};
        \item
            Because the fields $ \nabla h,Y $ in the Hodge-de Rham decomposition
            $ X=\nabla h+Y $ of $ X $ are orthogonal to one another in $
            L^2(M) $ we get that
            \begin{equation*}
                \int_{M^n}|\nabla h|^{2}dv_{g}=
                \int_{M^n}\langle\nabla h,\nabla h+Y\rangle dv_{g}=
                \int_{M^n}\langle\nabla h, X\rangle dv_{g}\leqslant0,
            \end{equation*}
            the inequality being a part of the hypothesis. Then, $ \nabla h=0 $
            and $ X=Y $. Since $Y$ is divergence free, we conclude that
            \[
                n\left(\log\frac{\sigma_{k}}{\sigma_{l}}-\lambda\right)=
                \divergence{X}=0,
            \]
            and, as such, the soliton is trivial.
    \end{enumerate}
    This proves the Theorem.
\end{myproof}

\begin{myproof}{T4}
    As we already know, the fundamental equation
    \[
        \dfrac{1}{2}\mathcal{L}_Xg=
        \left(\log{\dfrac{\sigma_k}{\sigma_l}}-\lambda\right)g,
    \]
    leads to
    \begin{equation}\label{trace}
        \divergence{X}=n\left(\log \dfrac{\sigma_{k}}{\sigma_{l}}-\lambda\right),
    \end{equation}
    and because we suppose that $ \mathcal{L}_Xg\geqslant0 $ we must then admit
    that $ \log{\frac{\sigma_k}{\sigma_l}}-\lambda\geqslant0 $. So, if we now
    take a cut-off function $\psi:M\rightarrow\mathbb{R}$ satisfying
    \begin{equation*}
        0\leqslant\psi\leqslant1\mbox{ on } M,\quad
        \psi\equiv 1\hspace{0,2cm}\text{in}\hspace{0,2cm} B_{r}(x_{0}),\quad
        \textup{supp}{(\psi)}\subset B_{2r}(x_{0})\quad \text{and}\quad
        |\nabla\psi|\leqslant\frac{K}{r},
    \end{equation*}
    where $ K>0 $ is a real constant, we are in place to conclude that
    \begin{equation*}
        \begin{split}
            n\int_{B_{r}(x_{0})}
            \left(
            \log\frac{\sigma_{k}}{\sigma_{l}}-\lambda
            \right)
            dv_{g}
            &=
            \int_{B_r(x_0)}
            n\psi\left(\log{\frac{\sigma_k}{\sigma_l}}-\lambda\right)
            dv_g
            \\
            &\leqslant
            \int_{B_{2r}(x_0)}
            n\psi\left(\log{\frac{\sigma_k}{\sigma_l}}-\lambda\right)
            dv_g
            \\
            &=
            \int_{B_{2r}(x_{0})}\psi\,\divergence{X}dv_{g}
            \\
            &=
            -\int_{B_{2r}(x_{0})}g(\nabla\psi,X)dv_{g}
            \\
            &
            \leqslant
            \int_{B_{2r}(x_{0})\setminus B_{r}(x_{0})}|-\nabla\psi||X|dv_{g}
            \\
            &
            \leqslant
            K\int_{B_{2r}(x_{0})\setminus B_{r}(x_{0})}\frac{|X|}{r}dv_{g},
            \\
            &
            \leqslant
            2K\int_{M\setminus B_r(x_0)}\dfrac{|X|}{d(x,x_0)}dv_g,
        \end{split}
    \end{equation*}
    from what it follows that
    \begin{align*}
        0\leqslant
        \int_M
        \left(
        \log\frac{\sigma_{k}}{\sigma_{l}}-\lambda
        \right)
        dv_{g}
        &=
        \lim_{r\to\infty}
        \int_{B_{r}(x_{0})}
        \left(
        \log\frac{\sigma_{k}}{\sigma_{l}}-\lambda
        \right)
        dv_{g}
        \\
        &\leqslant
        \frac{2K}{n}\lim_{r\to\infty}
        \int_{M\setminus B_r(x_0)}\frac{|X|}{d(x,x_0)}dv_g=0.
    \end{align*}
    Henceforth, we have that $
    \mathcal{L}_Xg=\log{\frac{\sigma_k}{\sigma_l}}-\lambda=0 $ which proves the
    Theorem.
\end{myproof}

\begin{myproof}{Teoo1}
    It follows from Lemma \ref{lema} $\textup{(c)}$ that if the scalar curvature
    of $ (M^n,g,\nabla f,\lambda) $ is a constant function on $ M^n $, then
    \begin{equation}\label{eq:-1}
        \Delta
        \left(
        \log\dfrac{\sigma_{k}}{\sigma_{l}}-\lambda
        \right)+
        \frac{R}{n-1}
        \left(
        \log\dfrac{\sigma_{k}}{\sigma_{l}}-\lambda
        \right)=0,
    \end{equation}
    and, by the min-max principle, we must have $ R>0 $. Because we know that
    \begin{equation}\label{eq:-2}
        \Delta f=n
        \left(
        \log\dfrac{\sigma_{k}}{\sigma_{l}}-\lambda
        \right),
    \end{equation}
    we then get
    \[
        \Delta
        \left(
        \log\dfrac{\sigma_{k}}{\sigma_{l}}-\lambda+\frac{R}{n(n-1)}f
        \right)
        =0,
    \]
    and since $ (M^n,g) $ is a compact Riemannian manifold, we see that
    \[
        \log\dfrac{\sigma_{k}}{\sigma_{l}}-\lambda+\frac{R}{n(n-1)}f=c
        \quad\mbox{on}\quad M^n,
    \]
    for a certain $ c\in\mathbb{R} $, by a Theorem of E. Hopf. But, then
    \[
        \nabla
        \left(
        \log\dfrac{\sigma_{k}}{\sigma_{l}}-\lambda
        \right)
        +
        \frac{R}{n(n-1)}\nabla f
        =0,
    \]
    and so
    \[
        \nabla_{X}\nabla
        \left(
        \log\dfrac{\sigma_{k}}{\sigma_{l}}-\lambda
        \right)
        =
        -\frac{R}{n(n-1)}\nabla_{X}\nabla f
        =
        -\frac{R}{n(n-1)}
        \left(
        \log\frac{\sigma_{k}}{\sigma_{l}}-\lambda
        \right)X.
    \]
    We can now apply Theorem A from Obata \cite{obata1962certain} to conclude
    that $(M^n,g)$ is isometric with an Euclidean sphere of radius $ \sqrt{r} $,
    $r=R/n(n-1)$. To prove our last claim we notice that we can assume that
    $R=n(n-1)$ possibly at the cost of rescaling the metric $ g $. From
    equations \eqref{eq:-1} and \eqref{eq:-2} it's seen that $\frac{\Delta
    f}{n}$ is an eigenfunction of the Laplacian on $(\mathbb{S}^n,g)$ and so
    there must exist a $v\in\mathbb{S}^n$ such that $\frac{1}{n}\Delta
    f=h_{v}=-\frac{1}{n}\Delta h_{v}$. Hence, $\Delta (f+ h_{v})=0$ but then
    $f=h_{v}+c$ for some real $c$.
\end{myproof}

\begin{myproof}{teosigmak}
    By Theorem 1.1 of \cite{catino2012global} the only nontrivial compact
    quotient gradient almost Yamabe solitons reside in the conformal class of
    the Euclidean sphere and because of that we can assume that
    \[
        M^n=\mathbb{S}^n\quad\mbox{and}\quad\varphi^{-2}g=g_{\mathbb{S}^n},
    \]
    where $ \varphi\in C^\infty(\mathbb{S}^n) $ is strictly positive. Then, the
    Ricci tensors of $ g $ and $ g_{\mathbb{S}^n} $ are correlated by the
    equation \cite{besse2007einstein}
    \[
        \Ric_{\mathbb{S}^n}
        =
        \Ric_{g}
        +
        \frac{1}{\varphi^2}
        \big{\{}
        (n-2)\varphi\nabla^{2}\varphi
        +
        [\varphi\Delta\varphi-(n-1)|\nabla \varphi|^2]g
        \big{\}},
    \]
    which we algebraically manipulate in order to get the similar equation
    \begin{equation}\label{eq:conformal-schouten-sphere}
        A_{g_{\mathbb{S}^n}}
        =
        A_{g}
        +
        \frac{\nabla^2\varphi}{\varphi}
        -
        \frac{1}{2}\frac{|\nabla\varphi|^2}{\varphi^2}g,
    \end{equation}
    for the Schouten tensors. But then we have
    \[
        \frac{1}{2}
        \left(
        \varphi^2+\frac{|\nabla\varphi|^2}{\varphi^2}
        \right)g=
        A_{g}+\frac{\nabla^2\varphi}{\varphi},
    \]
    from what it follows that
    \begin{equation}\label{hessiana}
        \nabla^2\varphi=
        \varphi
        \left[
            -A_g+
            \frac{1}{n}\left(\sigma_1(g)+\frac{\Delta\varphi}{\varphi}\right)g
            \right].
    \end{equation}
    Notice that Lemma \ref{KW1} applied to $
    T=T_k(g^{-1}A_g) $ and $ X=\nabla\varphi $ gives
    \begin{equation}\label{inte}
        \int_{M}\langle T_{k}(g^{-1}A_{g}),\nabla^{2}\varphi\rangle dv_{g}=0,
    \end{equation}
    because $ \trace{T_k(g^{-1}A_g)}=(n-k)\sigma_k(g) $ is constant on $
    \mathbb{S}^n $ by hypothesis and $ \divergence{T_k(g^{-1}A_g)}=0 $. A
    combination of \eqref{inte} and \eqref{hessiana} above leads to
    \begin{align*}
        0&=
        \int_{M}
        \langle
        T_{k}(g^{-1}A_{g}),
        -\varphi A_{g}+
        \frac{\sigma_1(g)\varphi+\Delta\varphi}{n}g
        \rangle
        dv_{g}=0
        \\
        &=
        \int_M
        \left[
            -\varphi\langle T_k(g^{-1}A_g),A_g\rangle+
            \frac{\sigma_1(g)\varphi+\Delta\varphi}{n}
            \langle T_k(g^{-1}A_g),g\rangle
            \right]
        dv_g
        \\
        &=
        \int_M
        \varphi
        \left[
            \left(\frac{n-k}{n}\right)\sigma_1(g)\sigma_k(g)-
            (k+1)\sigma_{k+1}(g)
            \right]
        dv_g
    \end{align*}
    where we have used the identity $\trace{T_k(g^{-1}A_g\circ
    A_g)}=(k+1)\sigma_{k+1}(g)$ \cite{reilly1974hessian}. By Lemma 23 of
    \cite{viaclovsky2000conformal} we conclude that
    \[
        \left(\frac{n-k}{n}\right)\sigma_{1}\sigma_{k}=(k+1)\sigma_{k+1},
    \]
    implying that $(\mathbb{S}^n,g)$ is an Einstein manifold. In particular, the
    scalar curvature of $ g $ is constant on $ \mathbb{S}^n $ and by Theorem
    \ref{Teoo1} there is even an isometry between $(\mathbb{S}^n,g)$ and $
    (\mathbb{S}^n,g_{\mathbb{S}^n}) $ which proves the Theorem.
\end{myproof}

\begin{myproof}{final}
    Lemma \ref{naza} applied to the data $T=\overset{\circ}{\Ric_g}$, $X=\nabla
    f$ and $\varphi=f$ gives
    \begin{equation}\label{eqsei}
        \divergence{\overset{\circ}{\Ric_g}(f\nabla f)}
        =
        f(\divergence{\overset{\circ}{\Ric_g})(\nabla f)}
        +
        f \langle\nabla^2 f, \overset{\circ}{\Ric_g}\rangle
        +
        \overset{\circ}{\Ric_g}(\nabla f.\nabla f),
    \end{equation}
    and it then follows from the second contracted Bianchi identity that
    \begin{equation}\label{eqsei1}
        (\divergence{\overset{\circ}{\Ric_g})(\nabla f)}
        =
        \frac{n-2}{2n}\langle\nabla f,\nabla R\rangle.
    \end{equation}
    A straightforward computation shows that
    \begin{equation}\label{eqsei2}
        f \langle\nabla^2 f,\overset{\circ}{\Ric_g}\rangle
        =
        f\left(
        \log\frac{\sigma_{k}}{\sigma_{l}}-\lambda
        \right)
        \langle g,\overset{\circ}{\Ric_g}\rangle
        =0,
    \end{equation}
    and equations \eqref{eqsei}, \eqref{eqsei1} and \eqref{eqsei2} together give
    \begin{equation}\label{eqsei3}
        \frac{1}{2}\divergence{\overset{\circ}{\Ric_g}(\nabla f^2)}
        =
        \frac{n-2}{4n}\langle\nabla R_g,\nabla f^2\rangle
        +
        \overset{\circ}{\Ric_g}(\nabla f,\nabla f).
    \end{equation}
    Proposition 1 of \cite{caminha2010complete} tell us that
    $\divergence{\overset{\circ}{\Ric_g}(\nabla f^2)}=0$ because
    $|\overset{\circ}{\Ric_g}(\nabla f^2)|\in L^{1}(M)$. Consequently,
    \[
        \langle\nabla R_g,\nabla f^2\rangle=0\quad\mbox{and}\quad
        \overset{\circ}{\Ric_g}(\nabla f,\nabla f)=0.
    \]
    As $(M^n,g,\nabla f,\lambda)$ is a nontrivial quotient gradient almost
    Yamabe soliton, any regular level set $\Sigma$ of the potential function $f$
    admits a maximal open neighborhood $U\subset M$ in which $ g $ can be
    written like
    \begin{equation}\label{dfdf}
        g=dr\otimes dr+(f'(r))^2g^{\Sigma},
    \end{equation}
    where $g^{\Sigma}$ is the restriction of $g$ to $\Sigma$ (see
    \cite{catino2012global}). Since $M$ is noncompact, $f$ has at most one
    critical point. As the Ricci tensor of a warped product metric, $ \Ric_g $
    now admits the following decomposition
    \begin{equation}\label{ricci}
        \Ric_g=
        \Ric^{\Sigma}
        -
        (n-1)\frac{f^{'''}}{f^{'}}dr\otimes dr
        -
        [(n-2)(f'')^2+f'f''']g^{\Sigma},
    \end{equation}
    thus giving $\frac{R_g}{n}=-(n-1)\frac{f'''}{f'}$ because
    $\overset{\circ}{\Ric_g}(\nabla f,\nabla f)=0$. Equation \eqref{ricci} can
    also be manipulated to show that
    \[
        \Ric_g(\nabla f)=\frac{R_g}{n}\nabla f,
    \]
    of which
    \begin{equation}\label{nabla}
        \nabla
        \left(
        \log\dfrac{\sigma_{k}}{\sigma_{l}}-\lambda
        \right)
        +
        \frac{R_g}{n(n-1)}\nabla f
        =0,
    \end{equation}
    is a consequence by Lemma \ref{lema} $ \textup{b)} $. The divergence of
    equation \eqref{nabla} is
    \begin{equation}\label{lapla}
        \Delta
        \left(
        \log\dfrac{\sigma_{k}}{\sigma_{l}}
        -\lambda
        \right)
        +
        \frac{1}{n(n-1)}
        \langle\nabla R_g,\nabla f\rangle
        +
        \frac{R_g}{n-1}
        \left(
        \log\dfrac{\sigma_{k}}{\sigma_{l}}
        -
        \lambda
        \right)
        =0.
    \end{equation}
    which we compare with the expression in Lemma \ref{lema} $ \textup{c)} $ to
    see that $\langle\nabla R_g,\nabla f\rangle=0$. Since $R_g$ only depends on
    $r$ we get that
    \[
        f'R_g'
        =
        f'\langle\nabla R_g,\partial r\rangle
        =
        \langle\nabla R_g,\nabla f\rangle
        =0,
    \]
    implying that the scalar curvature $R_g=R$ is constant. We claim that
    $R\leqslant0$. As a matter of fact, if we had $R>0$, we would then have
    \[
        \Ric_g\geqslant\frac{R}{n}g>\frac{R}{2n}g,
    \]
    because $\overset{\circ}{\Ric_g}\geqslant0$ by hypothesis and the manifold $
    M^n $ would then be compact by the Bonnet-Myers Theorem. Therefore, $
    R\leqslant0 $.
    \begin{enumerate}[a)]
        \item
            It follows from \eqref{nabla} that $
            \log{\frac{\sigma_k}{\sigma_l}}-\lambda=c $ for some $
            c\in\mathbb{R} $ because we now have $ R=0 $. By Theorem 2 of
            \cite{tashiro1965complete} $ (M^n,g) $ must be isometric with flat
            Euclidean space $ \mathbb{R}^n $ in case $ c\neq0 $. Since this
            would leave us with $ \sigma_1(g)=\sigma_2(g)=\cdots=\sigma_n(g)=0
            $, the function $ \log{\frac{\sigma_k}{\sigma_l}} $ could not be
            defined. Then, $ c=0 $ and so $ \nabla^2 f=0 $ by the fundamental
            equation \eqref{eq fundamental}. Theorem B of Kanai
            \cite{kanai1983differential} then implies that $ (M^n,g) $ is
            isometric with a Riemannian product manifold $
            \mathbb{R}\times\mathbb{F}^{n-1} $. Notice that $ \Ric_g\geqslant0 $
            forces $ \mathbb{F}^{n-1} $ to have a nonnegative Ricci curvature;
        \item
            If $f$ has no critical points and $ R<0 $ then once more by
            \eqref{nabla} we get that
            $\log\frac{\sigma_{k}}{\sigma_{l}}-\lambda$ is not constant on $ M^n
            $ and satisfies
            \[
                \nabla_{X}\nabla
                \left(
                    \log\dfrac{\sigma_{k}}{\sigma_{l}}-\lambda
                \right)
                =
                -\frac{R}{n(n-1)}\nabla_{X}\nabla f
                =
                -\frac{R}{n(n-1)}
                \left(
                    \log\dfrac{\sigma_{k}}{\sigma_{l}}-\lambda
                \right)X,
            \]
            on $ M^n $ for every $ X\in\mathfrak{X}(M) $. In virtue of Theorem D
            in \cite{kanai1983differential} the manifold $(M^n,g)$ is isometric
            with a warped product manifold
            $(\mathbb{R}\times\mathbb{F}^{n-1},dr^2+\xi(r)^2g_\mathbb{F})$ in
            which the warping function $\xi$ solves the second order linear ODE
            with constant coefficients $\xi''+\frac{R}{n(n-1)}\xi=0$;
        \item
            In our last call to equation \eqref{nabla} we observe that if $f$
            has exactly one critical point and $ R<0 $ then
            $\log\frac{\sigma_{k}}{\sigma_{l}}-\lambda$ is not constant on $ M^n
            $ and must satisfy
            \[
                \nabla_{X}\nabla
                \left(
                    \log\dfrac{\sigma_{k}}{\sigma_{l}}-\lambda
                \right)
                =
                -\frac{R}{n(n-1)}\nabla_{X}\nabla f
                =
                -\frac{R}{n(n-1)}
                \left(
                    \log\dfrac{\sigma_{k}}{\sigma_{l}}-\lambda
                \right)X,
            \]
            on $ M^n $ for every $ X\in\mathfrak{X}(M) $. We then apply Theorem
            C in \cite{kanai1983differential} to conclude that $(M^n,g)$ is
            isometric with a hyperbolic space.
    \end{enumerate}
\end{myproof}

\bibliographystyle{abbrv}
\bibliography{refs}

\end{document}